\documentclass[bj,nameyear]{imsart}
\usepackage{amsfonts}
\usepackage{amssymb,amsmath,amsthm}
\usepackage{mathrsfs}
\usepackage{natbib}
\usepackage{bbm}
\RequirePackage[colorlinks,citecolor=blue,urlcolor=blue]{hyperref}

\startlocaldefs
\numberwithin{equation}{section}
\theoremstyle{plain}

\newtheorem{theorem}{Theorem}[section]
\newtheorem{lemma}[theorem]{Lemma}
\newtheorem{proposition}[theorem]{Proposition}

\theoremstyle{remark}
\newtheorem{remark}[theorem]{Remark}

\numberwithin{equation}{section}

\begin{document}
\begin{frontmatter}
\title{On large deviations and intersection of random interlacements}
\runtitle{On large deviations and intersection of random interlacements}

\begin{aug}
\author[A]{\fnms{Xinyi}~\snm{Li}\ead[label=e1]{xinyili@bicmr.pku.edu.cn}}
\and
\author[B]{\fnms{Zijie}~\snm{Zhuang}\ead[label=e2]{zijie123@wharton.upenn.edu}}

\address[A]{BICMR, Peking University\printead[presep={,\ }]{e1}}

\address[B]{University of Pennsylvania\printead[presep={,\ }]{e2}}
\end{aug}

\begin{abstract}
We investigate random interlacements on $\mathbbm{Z}^d$ with $d \geq 3$, and derive the large deviation rate for the probability that the capacity of the interlacement set in a macroscopic box is much smaller than that of the box. As an application, we obtain the large deviation rate for the probability that two independent interlacements have empty intersections in a macroscopic box. We also prove that conditioning on this event, one of them will be sparse in the box in terms of capacity. This result is an example of the entropic repulsion phenomenon for random interlacements.
\end{abstract}



\end{frontmatter}
\section{Introduction}
\label{sec:intro}
The model of random interlacements is introduced in \cite{S10} to understand the trace of simple random walks as well as percolation with long-range correlations. Apart from its percolative properties (see in particular \cite{S10, SS09}), the large deviation properties of the model along with the trace of simple random walk, is also a central object of study; see \cite{S17, LS14, L17, S19a, S19b, NS20, CN20b, S21a, S21b} for recent progresses in this direction.

In this article, we consider the intersection of two independent interlacements. Percolative properties of the intersection set (and its complement) are considered by the second author in \cite{Z20}. In order to better characterize the intersection set, we calculate the asymptotic probability it leaves a macroscopic hole in the space. In doing so we also show that the optimal strategy for two interlacements to avoid each other is to force the one with smaller intensity to be almost empty in the box in terms of capacity, which falls within the scope of the entropic repulsion phenomenon for interlacements. (This phenomenon in the setup of Gaussian free fields is first studied in \cite{BDZ95, DG99} and then on interlacements recently in \cite{CN20b}.) In the course of the proof, an important step is to compute the large deviation rate for the probability that the capacity of the interlacement set in a macroscopic box is much smaller than that of the box. This rate function turns out to be given by a constraint problem on capacity, which is of independent interest. 

We now describe our results in more detail. Consider $\mathbbm{Z}^d$ and $\mathbbm{R}^d$ with $d \geq 3$, and two independent interlacements $\mathcal{I}_1^{u_1}, \mathcal{I}_2^{u_2}$ on $\mathbbm{Z}^d$ with intensity parameters $u_1, u_2$ respectively. We will omit the subscripts and write $\mathcal{I}^u$ when only one of them is considered. We write $\mathbbm{P}$ for the probability measure governing these objects, and $\mathbbm{E}$ for the corresponding expectation. Let $B(x,r)$ (resp.\ $\widetilde{B}(x,r)$) denote the closed $l^{\infty}$-norm box in $\mathbbm{Z}^d$ (resp.\ $\mathbbm{R}^d$) centered at $x$ and of radius $r$. We denote by ${\rm cap}(A)$ for $A \subset \subset \mathbbm{Z}^d$ the discrete capacity, and $\widetilde{{\rm cap}}(A)$ for $A \subset \mathbbm{R}^d$ the Brownian capacity. We call a set $A \subset \mathbbm{R}^d$ nice if it is the union of a finite number of boxes. For $\lambda \geq 0$, let $f(\lambda)$ be the solution of the following constraint problem:
\begin{equation}
\label{eq:def-f}
f(\lambda)=\inf_{A \;{\rm nice, }\;\widetilde{{\rm cap}} (A) \leq \lambda} \widetilde{{\rm cap}} (\widetilde{B}(0,1) \backslash A ).
\end{equation}
See Proposition~\ref{prop:property-f} for basic properties of $f$, 

Our first result is on the large deviation rate for the probability that the interlacement set in a macroscopic box has small capacity.
\begin{theorem}
\label{thm:1-1}
For any $u>0$ and $ 0<\lambda<\frac{1}{d}\widetilde{{\rm cap}} (\widetilde{B}(0,1) ) $,
$$
\lim_{N \rightarrow \infty} \frac{1}{N^{d-2}} \log \mathbbm{P}\left[ {\rm cap} \left(B(0,N) \cap \mathcal{I}^u\right) < \lambda N^{d-2}\right] =- \frac{u}{d}f(d\lambda)\,.
$$
\end{theorem}

Here, $N^{d-2}$ appears twice for it is the order of the discrete capacity of $B(0,N)$ (it also appears in related problems for Gaussian free fields, e.g., \cite{BD93, S15, CN20a}). The dimension $d$ appears in the rate function because the Brownian capacity is approximately $d$ times the discrete capacity, see Lemmas~\ref{lem:disc-cont} and \ref{lem:disc-cont2} for a more precise statement.

Our next result gives the large deviation rate for the probability that two independent interlacements have no intersections in a macroscopic box, and shows that conditioned on this event the one with smaller intensity parameter will be negligible in terms of capacity.

\begin{theorem}
\label{thm:1-2}Consider two independent random interlacements $\mathcal{I}^{u_1}_1$ and $\mathcal{I}^{u_2}_2$.
\begin{itemize}
\item[(1).]{For any $u_1,u_2>0$,
\begin{equation*}
\lim_{N \rightarrow \infty}\frac{1}{N^{d-2}} \log \mathbbm{P}\left[\mathcal{I}^{u_1}_1 \cap \mathcal{I}^{u_2}_2 \cap B(0,N)= \emptyset \right] =-\frac{\min \{u_1 , u_2\}}{d} \widetilde{{\rm cap}} (\widetilde{B}(0,1) )\,.
\end{equation*}}
\item[(2).]{
For any $u_1>u_2>0$ and $\epsilon>0$,
\begin{equation*}
\lim_{N \rightarrow \infty}\mathbbm{P}\left[{\rm cap} \left(B(0,N) \cap \mathcal{I}^{u_2}_2\right) < \epsilon N^{d-2}\big{|}\mathcal{I}^{u_1}_1 \cap \mathcal{I}^{u_2}_2 \cap B(0,N)= \emptyset\right] =1\,.
\end{equation*}}
\end{itemize}
\end{theorem}

In Theorem~\ref{thm:1-1} and Theorem~\ref{thm:1-2}, we only write down the case of $B(0,N)$, which is the blow-up of $\widetilde{B}(0,1)$, but analogous results should also hold for more general sets (with some regularity assumptions). We can also show that the one with smaller intensity parameter is negligible in the box in terms of local time. When $u_1 = u_2$, it is possible to prove that with asymptotically equal probability one of these interlacements sets is negligible in terms of local time. However, this proof requires another approach, which we will not include in this paper for brevity. See Remark~\ref{remark:local-time} for more discussions.

Theorem~\ref{thm:1-2} is an example of the entropic repulsion phenomenon for random interlacements. This phenomenon suggests that conditional on some rare results (such as disconnecting a box from far away \cite{S17, NS20, CN20b}, exceeding expected values in a box \cite{S19b, S21a, S21b}, or having macroscopic holes \cite{S19a}), random interlacements should stick to the configuration with the lowest energy; namely tilted interlacements with some non-homogeneous density. (Tilted interlacements are first introduced in \cite{LS14} and plays a major role in both the aforementioned work and \cite{L17} for the asymptotic lower bounds for disconnection probability.) This phenomenon is also conjectured to exist for simple random walks which can be seen as interlacements with intensity zero. However, only a few rigorous results have been proved for random interlacements or simple random walks; see \cite{CN20b, S19a}, and there are still many open problems. Similar phenomenon for Gaussian free fields is better understood since it possesses a nice domain Markov property which is absent in interlacements. Another important issue is that upper and lower bounds involving percolative properties of interlacements are usually given through differently defined critical thresholds, while the sharpness of phase transition is still open for interlacements (however, this is no longer a problem for Gaussian free fields as sharpness in this case has been verified in \cite{DCGRS20}).

We briefly describe the proof strategy of Theorem~\ref{thm:1-1} and Theorem~\ref{thm:1-2}. 

The lower bound of Theorem~\ref{thm:1-1} follows directly from the definition of $f$. The interlacements can be forced to stay within the blow-up of the set which solves the constraint problem, and then the lower bound of the rate function is given by the Brownian capacity of this set.

We now turn to the upper bound. Heuristically speaking, similar to the decomposition of Gaussian free fields, we can also decompose interlacements into a local part and a global part. Since the local part is approximately independent, it is more difficult to be tilted than the global part which has long-range correlations. Therefore, when considering the probability of some rare events, we can assume that the local part remains unchanged and only the global part is tilted. In this case, the rare event corresponds to some tilting of the global part. Since the global part has an integrable structure, the cost of this tilting can be calculated and leads to an upper bound on the large deviation rate function.

To make this intuition rigorous, we apply the coarse-graining procedure introduced in \cite{S15, S17}. More specifically, we partition the macroscopic box into mesoscopic boxes whose side length is chosen appropriately. By the soft local time technique in \cite{PT15, CGPV13}, we can show that with super-exponential probability, interlacements will behave regularly in most boxes, e.g., the fraction of occupied points is consistent with the  local time profile (not necessarily the intensity parameter). This corresponds to the intuition that the local part is much harder to be tilted. In our case, we call the interlacement set in a mesoscopic box {\it good} if it either has a small average local time or is ``visible'' for simple random walks, i.e., has capacity comparable to that of the box, see \eqref{eq:condition1} and \eqref{eq:condition2}. Under the constraint that most boxes are good, we can replace the event that the interlacement set in the macroscopic box has small capacity by a collection of events in each of which the interlacement set has a small average local time in many mesoscopic boxes and the total capacity of these boxes is above a certain value. The probability of each of these events can be bounded above through the Laplace transform of local times (see Proposition~\ref{lem:laplace-transf}). Adding all these inequalities together, we can obtain an upper bound for the large deviation probability which matches the lower bound in the principle order.

Theorem~\ref{thm:1-2} is a direct corollary of Theorem~\ref{thm:1-1} by observing that
\begin{equation}\label{eq:1-2quick}
\mathbbm{P} \Big[ \mathcal{I}^{u_1} \cap \mathcal{I}^{u_2} \cap B(0,N) = \emptyset \Big] =\mathbbm{E} \left[ e^{-u_1 \cdot {\rm cap}\left( \mathcal{I}^{u_2}  \cap B(0,N)\right)} \right] \,.
\end{equation}

In what follows, we list some of these open problems (and difficulties therein) concerning the entropic repulsion phenomenon for interlacements and related models.

The fundamental problem is to properly describe ``the convergence of some random set towards tilted interlacements''. In this paper, we prove convergence in terms of capacity and discussed the convergence in terms of local times. It would be great if they could be improved to convergence in terms of local (or mesoscopic) distributions, although both are still far from enough to fully characterize tilted interlacements.

The next problem is to tilt interlacements downwards. As mentioned in Remark 5.5 of \cite{S19b}, when studying large deviation problems in which the optimal strategy involves tilting interlacements downwards, one will encounter difficulties in estimating the rate function, especially in the case where the optimal strategy cannot be solved explicitly. Fortunately, in our case, this problem does not arise, as the optimal strategy of the problem we are considering is very simple.

We now discuss the independence under conditioning. For two independent interlacements, under the conditioning we discuss in this work they should behave like two tilted interlacements and more importantly remain asymptotically independent (analogous results should also hold for multiple independent interlacements). In this paper, we partially prove an example of this argument, but can one say more about it? To get a satisfactory result, it seems necessary to answer the first problem (since it is also not clear how to say several large random sets are asymptotically independent). 

Finally, a word on the entropic repulsion for the trace of simple random walks. To our knowledge, few rigorous results have been obtained for simple random walks in this direction. Even the rate function is unknown in some cases. One classical question on the large deviations of simple random walks (which is also relevant to polymer models) is what happens when the range of simple random walks is much smaller than expected. The large deviation rate function is obtained in \cite{BBH01, P12}, and the authors of \cite{BBH01} call the optimal strategy the ``Swiss cheese'', meaning that the range covered by the random walk looks like Swiss cheese (or more precisely Emmentaler cheese), viz.\ covering a positive fraction (but not all) of points in the space. It is conjectured that the ``Swiss cheese'' strategy actually forces the simple random walk to behave like tilted interlacements around some point; see \cite{AS17, AS20a} for some recent progresses.

This work is organized as follows. In Section \ref{sec:2}, we introduce our notation and setup, and recall a few useful results. Section \ref{sec:coarse} is dedicated to the coarse-graining strategy, which is a necessary step in obtaining the upper bound in Theorem \ref{thm:1-1}. We discuss the constraint problem and Brownian capacities in Section \ref{sec:4}. Finally, we wrap up the proof of both theorems and briefly remark an alternative approach in Section \ref{sec:5}.

Finally, let us explain our convention concerning constants. Constants like $ c, c', C, C'$ may change from place to place, while constants with subscripts like $c_1,C_1$ are kept fixed throughout the article. All constants may depend on $d$ implicitly. The dependence on additional variables will be marked at the first occurrence of each constant.

\section{Notation and some useful results}\label{sec:2}
In this section, we review definitions of simple random walks, Brownian motions and random interlacements, and collect some useful results about capacities and local times.

We begin with some notation. We consider $\mathbbm{Z}^d$ and $\mathbbm{R}^d$ with $d \geq 3$. For a real value $a$, let $\lfloor  a\rfloor$ denote the largest integer not greater than $a$. Let $| \cdot |_1$ (resp. $| \cdot |_\infty$) denote the $l^1$-norms (resp. $l^{\infty}$-norms) in both $\mathbb{Z}^d$ and $\mathbbm{R}^d$. In what follows, we will add a tilde ``$\;\widetilde{\;}\;$'' above objects in the continuum to distinguish them from their discrete counterparts.  We write $B(x,r)=\{ y \in \mathbbm{Z}^d:|x-y|_\infty \leq r \}$ for the closed $l^\infty$-ball in $\mathbbm{Z}^d$ centered at $x$ and of radius $r$. Given $A \subset \mathbbm{Z}^d$, we write $\partial_i A$ for its inner boundary and $\partial A$ for its outer boundary. Let $\widetilde{B}(x,r) = \{ y \in \mathbbm{R}^d: |x-y|_\infty \leq r \}$ denote the closed $l^{\infty}$-ball in $\mathbbm{R}^d$ centered at $x$ and of radius $r$. Let $\widetilde{A}(x,r,R) = \{ y \in \mathbbm{R}^d:r < |x-y|_{\infty} \leq R  \}$ denote the $l^{\infty}$-annulus in $\mathbbm{R}^d$ centered at $x$, of inner radius $r$ and outer radius $R$. Given $A \subset \mathbbm{R}^d$, we write $\partial A$ for its boundary. Here we slightly abuse the notation $|\cdot|_1$, $| \cdot |_{\infty}$ and $\partial$ in $\mathbbm{Z}^d$ and $\mathbbm{R}^d$ (and hopefully they will be clear in the context). For $A \subset \mathbbm{Z}^d$, let $\widetilde A$ denote its $\mathbbm{R}^d$-filling (which contains all the $l^\infty$-balls centered at the vertices in $A$ with radius $1/2$ in $\mathbbm{R}^d$). For a set $B$ in $\mathbbm{R}^d$ and any integer $N \geq 1$, let $B_N$ stand for the blow-up of $B$ in the discrete: $B_N = \{ x \in \mathbbm{Z}^d ;\; \inf_{y\in NB} |x-y|_\infty<1 \}$. 

We continue with continuous-time simple random walks in $\mathbbm{Z}^d$. Let $P_x$ denote the law of a continuous-time simple random walk $\{ X_t \}_{t \geq 0}$ on $\mathbbm{Z}^d$ with jump rate $1$ started at a vertex $x$. Given $A \subset\subset\mathbbm{Z}^d$, let $H_A$ denote the first time that $X_t$ hits $A$ and $T_A$ denote the first time that $X_t$ leaves $A$. We write $e_A$ for the equilibrium measure of $A$, $\bar{e}_A$ for the normalized equilibrium measure and $\mbox{cap} (A)=\sum_x e_A(x)$ for the discrete capacity of $A$. We write $g(x,y)$ for the Green function with respect to simple random walks. 

We now review some notation in $\mathbbm{R}^d$ about standard Brownian motions. In this paper, all sets we consider in $\mathbbm{R}^d$ are open or closed sets. We use $|\cdot|$ to denote the volume ($d$-dimensional Lebesgue measure) of these sets. Let $W_x$ denote the law of a standard Brownian motion $\{ W_t\}_{t \geq 0}$ in $\mathbbm{R}^d$ started at a point $x$. Given $A \subset \mathbbm{R}^d$, let $\widetilde{H}_A$ denote the entrance time of $A$ and $\widetilde{T}_A$ denote the exit time from $A$. We write $\widetilde{e}_A$ for the equilibrium measure of $A$ and $\widetilde{{\rm cap}}(A)$ for the Brownian capacity of $A$. We write $\widetilde{g}(x,y)$ for the Green function with respect to Brownian motions. We call a set \textbf{regular} if its closure and interior have the same Brownian capacity. We call a set $A \subset \mathbbm{R}^d$ \textbf{nice} if it is the union of a finite number of boxes.

From now on, let $N>100^d$ denote the side length of the macroscopic box and $L=L(N)$ denote the size of the mesoscopic boxes, viz.\ $L$-boxes defined below. When a large $N$ is given, we choose $L$ as
\begin{equation}
\label{eq:def-L}
L = \left\lfloor N^{2/d}(\log N)^{1/d}\right\rfloor\,.
\end{equation}
Let $\delta>0$ be a constant governing all the errors which tends to zero in the end and pick a large integer $K=K(\delta)>100$  
(the choice of $K$ is given in \eqref{eq:par-choice}).

We call an $l^\infty$-ball $B(x,r)$ in $\mathbbm{Z}^d$ an \textbf{$L$-box} if $r=L$, $x \in (2K+1)L \mathbbm{Z}^d$ and $B(x,r) \subset B(0,N)$. For $x \in (2K+1)L \mathbbm{Z}^d$, we further write $B_x$ short for $B(x,L)$ and $D_x$ short for $B(x,KL)$. Then, $D_x$ are disjoint from each other. We will write $\mathcal{C}$ for the union of a collection of $L$-boxes and use ${\rm Card}(\mathcal{C})$ to denote the number of $L$-boxes it contains (note that this is different from the common usage of ${\rm Card}(\cdot)$). We sometimes also view it as a set in $\mathbbm{R}^d$.

The following lemma is part of Proposition 2.5 in \cite{S17}. It states that when $K$ is large (equivalently $L$-boxes are far apart from each other), the relative equilibrium measure defined on $\mathcal{C}$ is close to the equilibrium measure defined on each $L$-box. We omit the proof.

\begin{lemma}
\label{lem:compare-eq}
If $L \geq 1$ and $K \geq c_1(\delta)$, then for any $\mathcal{C}$ which is the union of a collection of $L$-boxes, any $L$-box $B$ contained in $\mathcal{C}$ and any $x \in B$
$$
(1-\delta)\bar{e}_B(x) \leq \frac{e_{\mathcal{C}}(x)}{e_\mathcal{C}(B)} \leq (1+\delta)\bar{e}_B(x) \,,
$$
where $e_\mathcal{C}(B)=\sum_{y \in B}e_\mathcal{C}(y)$. 
\end{lemma}

The next lemma is collected from Proposition A.1 in \cite{NS20}. It states that the discrete capacity of $\mathcal{C}$ is close to $d$ fraction of the Brownian capacity of $\widetilde{\mathcal{C}}$ when $L$ (equivalently $N$) and $K$ are both large. We also omit the proof.

\begin{lemma}
\label{lem:disc-cont}
If $L \geq c_2(\delta)$ and $K \geq c_3(\delta)$, then for any $\mathcal{C}$ which is the union of a collection of $L$-boxes and its $\mathbbm{R}^d$-filling $\widetilde{\mathcal{C}}$
\begin{equation*}
(1-\delta) \widetilde{{\rm cap}}(\widetilde{\mathcal{C}}) \leq d\cdot{\rm cap}(\mathcal{C}) \leq (1+\delta)\widetilde{{\rm cap}}(\widetilde{\mathcal{C}})\,.
\end{equation*}
\end{lemma}

The next lemma is a classical fact that relates the Brownian capacity and the discrete capacity of its blow-up. It follows from estimates through variational characterizations of both types of capacity. See e.g., Lemma 2.2 in \cite{BD93} for the corresponding bounds. Although the original proof works for nice sets, it can also be extended to general regular sets by Proposition 1.13 in \cite{PS78}.
\begin{lemma}
\label{lem:disc-cont2}
Suppose $A$ is a regular set in $\mathbbm{R}^d$. Then,
$$
\lim_{N \rightarrow \infty} \frac{1}{N^{d-2}} {\rm cap}(A_N) = \frac{1}{d} \widetilde{{\rm cap}}(A)\,.
$$
\end{lemma}

In the rest of this section, we will briefly introduce the (continuous-time) random interlacements on $\mathbbm{Z}^d$. We refer readers to \cite{DRS14} for a detailed introduction of the model, and to e.g.\ \cite{LS14} for the construction of the continuous-time interlacements. Let $W$ denote the space of continuous-time doubly-infinite $l^1$-neighbor paths in $\mathbbm{Z}^d$, and let $W^*$ denote the quotient space of $W$ modulo time shift. We write $\pi$ for the quotient map from $W$ to $W^*$. We can define a Poisson point measure $\mu$ on $W^*$ characterized by the following property. Given $A \subset \subset \mathbb{Z}^d$, let $W^*_A$ denote the paths in $W^*$ that pass through $A$ and $\mu_A$ denote the restriction of $\mu$ on $W^*_A$. Then, we have
\begin{equation}
\label{eq:def-RI}
    \mu_A \overset{d}{=}\sum_{i=1}^{N_A}\delta_{\pi(X^i)}\,,
\end{equation}
where $N_A\sim{\rm Poisson}(u\cdot  {\rm cap}(A))$, and $\{X^i _t\}_{t \in \mathbbm{R}}$ are doubly-infinite paths on $\mathbbm{Z}^d$ in which $X^i_0$ is a random point in $A$ sampled according to the equilibrium probability measure $\bar{e}_A$. Given $X^i_0$, the process $\{X^i_t\}_{t \geq 0}$ from time $0$ is a continuous-time simple random walk, and the reversed process $\{X^i_t\}_{t < 0}$ is a continuous time simple random walk conditional that $\{ H_A=\infty \}$ and independent of the process from time $0$. Moreover, conditional on  $N_A$, all these $N_A$ paths are independent. The set of vertices occupied by at least one of these paths is called the interlacement set at level $u$, denoted by $\mathcal{I}^u$. 

For $x \in \mathbbm{Z}^d$, let $L^u(x)$ denote the local time of $\mathcal{I}^u$ at the vertex $x$. Given two functions $f,h:\mathbbm{Z}^d \rightarrow \mathbbm{R}$, we write $\langle f,h \rangle =\sum_{x\in \mathbbm{Z}^d} f(x) h(x)$ for their inner product. The following lemma gives the Laplace transform of $\langle L^u, e_A \rangle$ which helps to bound the local time of interlacements.

\begin{lemma}
\label{lem:laplace-transf}
For $A \subset\subset \mathbbm{Z}^d$ and $s<1$, 
$$
\mathbbm{E}\left[e^{s \langle L^u, e_A \rangle}\right]=\exp \left( \frac{us \cdot {\rm cap}(A) }{1- s}  \right) \,.
$$
\end{lemma}
\begin{proof}
For $|s|<1$, we can obtain this formula by taking $V(x)=s e_A(x)$ in $(2.40)$ of \cite{S17}. It can be extended to $s \in (-\infty, -1]$ by analytic extension.
\end{proof}

\section{The coarse-graining procedure}
\label{sec:coarse}
In this section, we introduce the coarse-graining procedure in \cite{S15} and the coupling of excursions in \cite{PT15, CGPV13}.

First, we define good boxes which are important objects in the coarse-graining arguments. We call an $L$-box $B_x$ \textbf{good} if either of the following two conditions holds. In particular, we will call it \\\textbf{Type-I good} if
\begin{equation}
\label{eq:condition1}
\max_{z \in \partial_i B_x} P_z\left[H_{\mathcal{I}^u \cap B_x} =\infty \right] <\delta \,,
\end{equation}
\textbf{Type-II good} if
\begin{equation}
\label{eq:condition2}
\langle L^u, \bar{e}_{B_x} \rangle <\delta u \,,
\end{equation}
and \textbf{bad} otherwise. The definition is similar in spirit to (3.11)-(3.13) in \cite{S15}. Roughly speaking, a good box is one in which the behavior of random interlacements matches its density (not necessarily equal to the intensity parameter). For a typical $L$-box, if the average value of $L^u$ in it is small, the condition \eqref{eq:condition1} holds. Otherwise, the average value of $L^u$ is not too small, and then the random interlacements will occupy a positive fraction of points and thus the condition \eqref{eq:condition2} holds with high probability. 

For $\rho>0$, define the event $\mathcal{A}$ by
\begin{equation}
\label{eq:def-a}
\mathcal{A} = \{ \mbox{there are at most }\rho (N/L)^d\mbox{ bad boxes in }B(0,N)\}\,.
\end{equation}In the next proposition, we will prove that for any $\rho>0$, with overwhelmingly high probability $\mathcal{A}$ happens (or equivalently, most $L$-boxes are good). 
\begin{proposition}
\label{prop:es-ac}
For all $\delta,\rho>0$ and $K\geq 100$
\begin{equation}\label{eq:es-ac}
\lim_{N \rightarrow \infty} \frac{1}{N^{d-2}} \log \mathbbm{P}[\mathcal{A}^c]=-\infty \,.
\end{equation}
\end{proposition}
Before giving the proof, we need some preparations.

Let $\mathbbm{Q}$ denote the law of a series of Poisson point processes indexed by $x \in (2K+1)L \mathbbm{Z}^d$. Given $x \in (2K+1)L \mathbbm{Z}^d$, let $(X^i)_{i \geq 1}$ be i.i.d. distributed excursions from $\partial_i B_x$ to $\partial D_x$ with the law of a simple random walk started from a point in $\partial_i B_x$ sampled according to $\bar e_{B_x}$ and ending upon leaving the box $D_x$. Let $n(t)$ denote the counting measure of the Poisson process whose intensity is ${\rm cap}(B_x)$. An important feature is that $(X^i)_{i \geq 1}$ and $n(t)$ are independent triples among different $x$'s.

Given $x \in (2K+1)L \mathbbm{Z}^d \cap B(0,N)$, let $Y^1, \ldots, Y^{N_u}$ denote the excursions from $\partial_i B_x$ to $\partial D_x$ induced by $\mathcal{I}^u$ where $N_u$ is the number of excursions (one can see (2.29), (2.41) and (2.42) in \cite{S15} for more precise definition). These excursions depend on the vertex $x$, but we omit it in the notation for simplicity. We write 
$$
\Theta=\Theta(L)={\rm cap}(B_x)
$$
for short. We know that there exists $a>1$ such that for all $y \in \partial_i B_x$ and $z \in (D_x)^c$
$$
\frac{1}{a} \bar{e}_{B_x} < P_z \left[H_{B_x}=H_y|H_{B_x}<\infty\right] < a \bar{e}_{B_x}\,.
$$

By the soft local technique (see Lemma 2.1 in \cite{CGPV13} and Section 5 in \cite{S15} for the specific case of interlacements), we can find a coupling of $\mathbbm{P}$ and $\mathbbm{Q}$ (which will be denoted as $\widehat{\mathbbm{Q}}$) such that
\begin{itemize}
\item If $n\Big( \frac{\delta u}{a^2}\Big) \geq \frac{\delta u}{a^3} \Theta$ and $N_u \leq \frac{\delta u}{a^3} \Theta$, we have
\begin{equation}
\label{eq:com-1}
\{Y^1, \ldots Y^{N_u} \} \subset \{X^1, \ldots, X^{n(\delta u/a)} \}\,.
\end{equation}
\item If $n\Big(\frac{\delta u}{a^4}\Big)  \leq \frac{\delta u}{a^3} \Theta$ and $N_u >\frac{\delta u}{a^3} \Theta$, we have
\begin{equation}
\label{eq:com-2}
\{Y^1, \ldots Y^{N_u} \} \supset \{X^1, \ldots, X^{n(\delta u/a^5)} \}\,.
\end{equation}
\end{itemize}

We are now ready to prove Proposition~\ref{prop:es-ac}.

\begin{proof}

For $x \in (2K+1)L \mathbbm{Z}^d \cap B(0,N)$, let $\mathcal{A}^1_x$ denote the event that all of the following inequalities hold:
\begin{equation}
\label{eq:3.1compare}
n\Big(\frac{\delta u}{a^4}\Big)  \leq \frac{\delta u}{a^3} \Theta\,; \quad n\Big(\frac{\delta u}{a}\Big) \leq \delta u \Theta\,;\quad  n\Big(\frac{\delta u}{a^4}\Big)  \leq \frac{\delta u}{a^3} \Theta \,; \quad n\Big( \frac{\delta u}{a^5}\Big) \geq \frac{\delta u}{a^6} \Theta\,.
\end{equation}
 Since $\mathbbm{E} [n(t)] =  \Theta t$, it is easy to see from Hoeffding's inequality and the fact that $\Theta \geq CL^{d-2}$ that 
\begin{equation}
\label{eq:prop3.1-1}
\mathbbm{Q}\big[(\mathcal{A}^1_x)^c\big] \leq \exp\big(-cL^{d-2}\big)\,.
\end{equation}
If the event $\mathcal{A}^1_x$ happens, we know from \eqref{eq:com-1} and \eqref{eq:com-2} that when $N_u \leq \frac{\delta u}{a^3} \Theta$
$$
\{ Y^1 , Y^2  , ... ,Y^{N_u} \} \subset \{ X^1 , X^2 , ..., X^{ \delta u \Theta } \}  \,,
$$
while when $N_u > \frac{\delta u}{a^3} \Theta$
$$
 \{ Y^1, Y^2 ,...,Y^{N_u} \} \supset \{ X^1 , X^2 , ..., X^{ \delta u \Theta/a^6  } \} \,.
$$
We write $M = \frac{\delta u}{a^6} \Theta$. Note that $M \geq cL^{d-2}$. Let $\mathcal{A}^2_x$ denote the event that $$\max_{z \in \partial_i B_x} P_z[H_{(X^1 \cup X^2 \cup \ldots \cup X^M) \cap B_x} = \infty] <\delta \,.$$
Then,
\begin{equation*}
\begin{split}
\mathbbm{Q}\big[(\mathcal{A}^2_x)^c\big] &\leq  \sum_{z \in \partial_i B_x} \mathbbm{Q}\left[P_z\left[H_{(X^1 \cup X^2 \cup \ldots \cup X^M) \cap B_x} =\infty\right]\geq \delta \right]\\
& \leq  \sum_{z \in \partial_i B_x} \frac{1}{\delta} \mathbbm{Q}\otimes P_z\left[H_{(X^1 \cup X^2 \cup \ldots \cup X^M) \cap B_x} =\infty \right]\\
&= \sum_{z \in \partial_i B_x} \frac{1}{\delta} \mathbbm{Q}\otimes P_z\left[Z \cap(X^1 \cup X^2 \cup \ldots \cup X^M) \cap B_x =\emptyset \right] \\
&= \sum_{z \in \partial_i B_x} \frac{1}{\delta}P_z[ Z \cap X^1 \cap B_x = \emptyset]^{M} \leq C L^{d-1} P_z[ Z \cap X^1 \cap B_x = \emptyset]^{cL^{d-2}}.
\end{split}
\end{equation*}
Here, $Z$ is the trajectory of a simple random walk starting from $z$. By Theorem $3.3.2$ in \cite{L91}, $ P_z[ Z \cap X^1 \cap B_x = \emptyset]\leq 1-c$ if $d=3$; $1- \frac{c}{\log L}$ if $d=4$ and $1- cL^{(4-d)/2}$ if $d \geq 5$. So,
\begin{equation}
\label{eq:prop3.1-2}
\lim_{L \rightarrow \infty} \frac{1}{\log(L)} \log \mathbbm{Q}\big[(\mathcal{A}^2_x)^c\big] =-\infty \,.
\end{equation}
Let $R = \delta u \Theta$. We consider $R$ independent simple random walk trajectories $X^1,\ldots, X^R$ and their local times $L^1(\cdot), \ldots, L^R(\cdot)$. Let $\mathcal{A}^3_x$ denote the event that 
$$
\sum_{i=1}^R \langle L^i , \bar{e}_{B_x} \rangle < \delta u \,.
$$
Now, we will show that
\begin{equation}
\label{eq:prop3.1-3}
\lim_{L \rightarrow \infty} \frac{1}{\log(L)} \log \mathbbm{Q}\big[(\mathcal{A}^3_x)^c\big] =-\infty \,.
\end{equation}
This can be derived directly from the fact that $\langle L^i, e_{B_x} \rangle $ is dominated by an exponential random variable with mean smaller than one.

With \eqref{eq:prop3.1-1}, \eqref{eq:prop3.1-2} and \eqref{eq:prop3.1-3} in hand, we now turn to \eqref{eq:es-ac}. Note that
$$
B_x \mbox{ is good on the event }\mathcal{A}^1_x \cap \mathcal{A}^2_x \cap \mathcal{A}^3_x \,.
$$
There are two cases.
\begin{itemize}
\item If $N_u \leq \frac{\delta u}{a^3} \Theta$, then by \eqref{eq:com-1} and \eqref{eq:3.1compare}
$$
\mathcal{I}^u \cap B_x \subset X^1 \cup \ldots \cup X^{n(\delta u/a)} \subset X^1 \cup \ldots \cup X^{R}.
$$
Together with the definition of $\mathcal{A}^3_x$, we know that in this case $\langle L^u, \bar e_{B_x} \rangle <  \delta u$.
\item If $N_u > \frac{\delta u}{a^3} \Theta$, then by \eqref{eq:com-2} and \eqref{eq:3.1compare}
$$
\mathcal{I}^u \cap B_x \supset X^1 \cup \ldots \cup X^{n(\delta u/a^5)} \supset X^1 \cup \ldots \cup X^{M}.
$$
Together with the definition of $\mathcal{A}^2_x$, we know that in this case $\max_{z \in \partial_i B_x} P_z[H_{\mathcal{I}^u \cap B_x } = \infty] <\delta$.
\end{itemize} 

Now, Proposition~\ref{prop:es-ac} just follows from  \eqref{eq:prop3.1-1}, \eqref{eq:prop3.1-2}, \eqref{eq:prop3.1-3} and the fact that the behavior of $L$-boxes are independent under the law $\mathbbm{Q}$ (here we will also need the fact that $N^2 \log L \geq cL^d$ from \eqref{eq:def-L}).
\end{proof}

\section{Basic properties of $f$ and Brownian capacity}\label{sec:4}
In this section, we study the properties of $f$ from the constraint problem (recall \eqref{eq:def-f} for its definition) and prove Proposition~\ref{prop:property-bc} which relates the Brownian capacity of the coarse-grained sets to their discrete capacity, which is important in the proof of Theorem~\ref{thm:1-1}. 

In the following proposition, we collect some basic properties of the function $f$.
\begin{proposition}
\label{prop:property-f} 
The function $f$ defined in \eqref{eq:def-f} has the following properties:
\begin{itemize}
\item[{\rm (1).}]$f(\lambda)$ is decreasing.
\item[{\rm (2).}]$f(0)=\widetilde{{\rm cap}}(\widetilde{B}(0,1))$, and $f(\lambda)=0$ for all $\lambda \geq \widetilde{{\rm cap}}(\widetilde{B}(0,1))$.
\item[{\rm (3).}]$f(\lambda) \geq \widetilde{{\rm cap}}(\widetilde{B}(0,1)) - \lambda$, for all $0 < \lambda < \widetilde{{\rm cap}}(\widetilde{B}(0,1))$.
\item[{\rm (4).}]$f(\lambda)$ is continuous.
\end{itemize}
\end{proposition}
\begin{proof}
Claims (1) and (2) are direct from the definition. Claim (3) follows from the sub-additivity of the Brownian capacity. Next, we prove Claim (4). Since $f$ is decreasing, (here we hold the convention that $f(\lambda)=\widetilde{{\rm cap}}( \widetilde B(0,1))$ when $\lambda<0$), it suffices to show that for any $\Delta>0$, there exists $\epsilon = \epsilon(\Delta)>0$ such that for all $\lambda \geq 0$
\begin{equation}
\label{eq:prop4.1-1}
f(\lambda) \geq f(\lambda -\epsilon) - \Delta\,.
\end{equation}
Fix $\Delta>0$ and a constant $\epsilon$ to be chosen. For all $\lambda\geq 0$, by the definition of $f$, there exists a nice set $A \subset \widetilde B(0,1)$ such that
$$
\widetilde {\rm cap}(A) \leq \lambda \quad \mbox{and} \quad \widetilde {\rm cap}(\widetilde B(0,1) \backslash A) \leq f(\lambda) +\frac{\Delta}{2}\,.
$$
In order to prove \eqref{eq:prop4.1-1}, it suffices to find a nice set $A' \subset A$ such that
\begin{equation}
\label{eq:prop4.1-2}
\widetilde {\rm cap}(A') \leq (\lambda-\epsilon)_+ \quad \mbox{and} \quad \widetilde {\rm cap}(\widetilde B(0,1) \backslash A') \leq f(\lambda) +\Delta\,.
\end{equation}
If $f(\lambda) +\Delta \geq \widetilde{{\rm cap}}(\widetilde B(0,1))$, we can take $A'=\emptyset$ and then \eqref{eq:prop4.1-2} holds. From now on, we assume that $f(\lambda) +\Delta < \widetilde{{\rm cap}}(\widetilde B(0,1))$. There exists a large integer $N=N(\Delta)$ such that
\begin{itemize}
\item For all $\frac{1}{N}$-box $E$, we have $\widetilde{\rm cap}(E) \leq \frac{\Delta}{2}$.
\item If $|K \cap E| \geq \frac{1}{2}|E|$ for all $\frac{1}{N}$-box $E$ on the boundary of $\widetilde B(0,1)$ (viz.\ $E \subset \widetilde B(0,1)$ and $E \cap \partial  \widetilde B(0,1) \neq \emptyset$), then $\widetilde{{\rm cap}}(K) \geq \widetilde{{\rm cap}}(\widetilde B(0,1)) - \frac{\Delta}{2}$.
\end{itemize}
The first claim follows form the scaling property of $\widetilde{\rm cap}(\cdot)$ and the second claim follows from a Wiener-type argument, see e.g.\ Theorem 2.2.5 in \cite{L91}.
Since $\widetilde {\rm cap}(\widetilde B(0,1) \backslash A) \leq f(\lambda) +\frac{\Delta}{2} < \widetilde{{\rm cap}}(\widetilde B(0,1)) - \frac{\Delta}{2}$, we can find a $\frac{1}{N}$-box $E$ on the boundary of $\widetilde B(0,1)$ such that 
$$
|(\widetilde B(0,1) \backslash A) \cap E| < \frac{1}{2}|E|\,,
$$
and so $|A \cap E| > \frac{1}{2}|E|$. Take $A' = A \backslash E$. Then,
$$
\widetilde{{\rm cap}}(\widetilde B(0,1) \backslash A') \leq \widetilde{{\rm cap}}(\widetilde B(0,1) \backslash A)+\widetilde {\rm cap}(E) \leq f(\lambda)+\Delta
$$
and there exists $\epsilon = \epsilon(\Delta)>0$ such that
\begin{align*}
\widetilde {\rm cap}(A) &= \int_{x \in \partial \widetilde B(0,2)} W_x [ \widetilde H_A <\infty] \widetilde e_{\widetilde B(0,2)} dx \\
&\geq \int_{x \in \partial \widetilde B(0,2)} \left( W_x [ \widetilde H_{A'} <\infty] + W_x [ \widetilde H_{A \cap E} <\infty,\widetilde H_{A'} = \infty ] \right)\widetilde e_{\widetilde B(0,2)} dx \\
&\geq \widetilde {\rm cap}(A') + \epsilon \,.
\end{align*}
The first equation is by Theorem 1.10 in \cite{PS78}. The last inequality is because $|A \cap E| > \frac{1}{2}|E|$ implies that the Brownian motion starting from $\partial \widetilde{B}(0,2)$ has a (uniform) positive probability to hit $\widetilde{B}(0,1)$ at the set $A\cap E$ and then escape to infinity without ever hitting $A'$. In other words, there exists a constant $c(\Delta)>0$ such that
$$
\inf_{x \in \partial \widetilde B(0,2)}W_x [ \widetilde H_{A \cap E} <\infty,\widetilde H_{A'} = \infty ] >c(\Delta)\, ,
$$ 
enabling us to pick $\epsilon$ as $c(\Delta) \cdot \widetilde {\rm cap}(\widetilde B(0,2))$. Therefore, $A'$ satisfies \eqref{eq:prop4.1-2} and Claim (4) holds.
\end{proof}
\begin{remark}
By the continuity of $f$, the infimum in \eqref{eq:def-f} can also be taken over all the sets in $\mathbbm{R}^d$ with $C^1$ boundary. Finding the minimizer set $A$ (if it exists, which we believe is the case) in \eqref{eq:def-f} is an interesting question in itself. To the best knowledge of the authors there is neither a back-of-the-envelop quick solution, nor a ready answer in the literature. We hope experts in variational analysis could answer it. It is equally interesting consider more general sets rather than $\widetilde{B}(0,1)$. We note that except in very special cases, e.g., the union of two touching balls with a specifically chosen $\lambda$, there seems to be no trivial answer either.
\end{remark}

We now turn to the Brownian capacity. The next proposition is important in the proof of Theorem~\ref{thm:1-1}.
\begin{proposition}
\label{prop:property-bc}
Given $\delta>0$ and $K \geq 1$, there exist $c_4(\delta,K)$ and $c_5(\delta,K)$ such that if $L>c_4(\delta,K), \rho<c_5(\delta,K), t>0$, then for any $\mathcal{C}_1$ and $\mathcal{C}_2$ which are two disjoint collections of $L$-boxes in $\widetilde B(0,N)$ satisfying ${\rm Card}(\mathcal{C}_1)+{\rm Card}(\mathcal{C}_2) \geq (1-\rho)\big(\frac{2N}{(2K+1)L}\big)^d$ and $\widetilde {\rm cap}(\mathcal C_1) \leq tN^{d-2}$, we have
\begin{equation}\label{eq:property-bc}
\widetilde {\rm cap}(\mathcal C_2) \geq (f(t) - \delta) N^{d-2}.
\end{equation}
\end{proposition}
\begin{proof}
Let $\epsilon = \epsilon(\delta,K)>0$ be a constant to be chosen. By similar arguments as Corollary 6.5.9 of \cite{LV10}, we can choose $N$ and $\rho$ such that 
$$
\mbox{for all } x \in \widetilde B(0,N)\,, \mbox{ either } W_x[\widetilde H_{\mathcal{C}_1} = \infty] <\epsilon \mbox{ or } W_x[\widetilde H_{\mathcal{C}_2} = \infty] <\epsilon\,.
$$
We now consider
$$
\mathcal{C}_1' = \mathcal{C}_1 \cup \{ x \in \widetilde B(0,N) : W_x[\widetilde H_{\mathcal{C}_1} = \infty] <\epsilon \}\;\;\mbox{and}\;\;
\mathcal{C}_2' = \mathcal{C}_2 \cup \{ x \in \widetilde B(0,N) : W_x[\widetilde H_{\mathcal{C}_2} = \infty] <\epsilon \} \,.
$$
Then, we have $\mathcal{C}_1' \cup \mathcal{C}_2' = \widetilde B(0,N)$. By Theorem 1.10 in \cite{PS78},
\begin{equation*}
\begin{split}
\widetilde{{\rm cap}}(\mathcal{C}_1') &= \int_{x \in \partial \widetilde B(0,N)} W_x[\widetilde{H}_{\mathcal{C}_1'} <\infty] \widetilde{e}_{\widetilde B(0,N)}(dx) \\
&=\int_{x \in \partial \widetilde B(0,N)} \left( W_x[\widetilde{H}_{\mathcal{C}_1} <\infty] +W_x[\widetilde{H}_{\mathcal{C}_1'}<\infty, \widetilde{H}_{\mathcal{C}_1}=\infty]  \right) \widetilde{e}_{\widetilde B(0,N)}(dx)\\
& \leq \int_{x \in \partial \widetilde B(0,N)}  \left( W_x[\widetilde{H}_{\mathcal{C}_1} <\infty] +W_x[\widetilde{H}_{\mathcal{C}_1'}<\infty] \sup_{y \in \mathcal{C}_1'}W_y[\widetilde{H}_{\mathcal{C}_1}=\infty]  \right) \widetilde{e}_{\widetilde B(0,N)}(dx)\\
&\leq \widetilde{{\rm cap}}(\mathcal{C}_1) +\epsilon \cdot \widetilde{{\rm cap}}(\mathcal{C}_1') \leq t N^{d-2} + \epsilon \widetilde {\rm cap}(\widetilde B(0,N))\,.
\end{split}
\end{equation*}
Similarly, we have
$$
\widetilde{{\rm cap}}(\mathcal{C}_2') \leq \widetilde{{\rm cap}}(\mathcal{C}_2) + \epsilon \widetilde {\rm cap}(\widetilde B(0,N))\,.
$$
The set $\mathcal{C}_1'$ can be approximated below by a sequence of nice sets $\{D_n \}$ and we can apply the capacity lower bound obtained from the function $f$ to $\{\widetilde B(0,1) \backslash D_n \}$ (and in a similar fashion for $\mathcal{C}_2'$). By the continuity of $f$, i.e.\ Claim (4) in Proposition~\ref{prop:property-f}, we can choose $\epsilon$ such that \eqref{eq:property-bc} is satisfied.
\end{proof}

\section{Proof of Theorems~\ref{thm:1-1} and \ref{thm:1-2}}
\label{sec:5}
In this section, we will prove Theorem~\ref{thm:1-1}. The lower bound can be derived from the definition of $f$ because we can force the interlacements to stay within the blow-up of the minimizer of the constraint problem in \eqref{eq:def-f} and the probability matches the large deviation rate. For the upper bound, we need to use the coarse-graining procedure introduced in Section~\ref{sec:coarse} and enumerate on all possible collections of type-II good boxes. We will consider a quantity $H$ defined in \eqref{eq:def-h} which encapsulates the information of the local time in type-II good boxes. The upper bound then follows from controls on $H$, in particular the Laplace transform in Lemma~\ref{lem:laplace-transf}.

We begin with the lower bound.
\begin{proof}[Proof of the lower bound in Theorem~\ref{thm:1-1}]
Fix $\lambda>0$ and $\alpha>f(d\lambda)$. By Claim (4) in Proposition~\ref{prop:property-f}, we can choose a nice set $A$ such that $$\widetilde{{\rm cap}}(A) < d\lambda \quad \mbox{and} \quad \widetilde{{\rm cap}}(\widetilde{B}(0,1) \backslash A ) < \alpha\,.$$ Let $B=\widetilde{B}(0,1) \backslash A$. Recall that $A_N$ and $B_N$ stand for the blow-up of $A$ and $B$ respectively. Then, $A_N \cup 
B_N \supset B(0,N)$. By Lemma~\ref{lem:disc-cont2},
\begin{equation*}
\lim_{N \rightarrow \infty} \frac{1}{N^{d-2}}{\rm cap}(A_N) = \frac{1}{d}  \widetilde{{\rm cap}}(A) < \lambda \,,\quad\mbox{and}\quad\lim_{N \rightarrow \infty} \frac{1}{N^{d-2}}{\rm cap}(B_N) = \frac{1}{d}  \widetilde{{\rm cap}}(B) < \frac{\alpha}{d} \,.
\end{equation*}
So,
\begin{equation*}
\begin{split}
& \quad \liminf_{N \rightarrow \infty} \frac{1}{N^{d-2}} \log \mathbbm{P}[ {\rm cap} (B(0,N) \cap \mathcal{I}^u) < \lambda N^{d-2}] \geq \liminf_{N \rightarrow \infty} \frac{1}{N^{d-2}} \log \mathbbm{P}[ \mathcal{I}^u \cap B(0,N) \subset A_N ] \\
&\geq \liminf_{N \rightarrow \infty} \frac{1}{N^{d-2}} \log \mathbbm{P}[ \mathcal{I}^u \cap B_N =\emptyset ] =  \liminf_{N \rightarrow \infty} \frac{1}{N^{d-2}} \log  \left[ \exp (-u \cdot {\rm cap} (B_N)) \right]  >- \frac{u \alpha}{d} \,.
\end{split}
\end{equation*}
This holds for any $\alpha>f(d\lambda)$ and thus
\begin{equation*}
\liminf_{N \rightarrow \infty} \frac{1}{N^{d-2}} \log \mathbbm{P}[ {\rm cap} (B(0,N) \cap \mathcal{I}^u) < \lambda N^{d-2}] \geq -\frac{u}{d}f(d\lambda) \,.
\end{equation*}
This finishes the proof of the lower bound.
\end{proof}
We now turn to the upper bound. Recall \eqref{eq:def-a} that $\mathcal{A}$ represents the event that there are at most $\rho(N/L)^d$ bad boxes in $B(0,N)$. We write 
$$
\mathcal{B}=\big\{{\rm cap}[\mathcal{I}^u \cap B(0,N)] <\lambda N^{d-2}\big\}\,.
$$ Fix $0<\delta<1$.  Pick a large $N$ (and recall \eqref{eq:def-L} for its relation with $L$), $K$ and some $\rho$ such that
\begin{equation}
\label{eq:par-choice}
K \geq c_1(\delta) \vee c_3(\delta) \,; \quad L \geq c_2(\delta) \vee c_4(\delta,K) \,; \quad 0< \rho <c_5(\delta,K)\,.
\end{equation}
We now define a quantity which is abnormally small under the event $\mathcal{A} \cap \mathcal{B}$. Let $\mathcal{C}_1$ denote the union of Type-I good boxes and $\mathcal{C}_2$ denote the union of Type-II good boxes (recall \eqref{eq:condition1} and \eqref{eq:condition2} for the definition). Define $H$ as
\begin{equation}
\label{eq:def-h}
H = \langle L^u, e_{\mathcal{C}_2} \rangle \,.
\end{equation}
\begin{lemma} On the event $\mathcal{A} \cap \mathcal{B}$, we have
\begin{equation}
\label{eq:es-c2h}
{\rm cap}(\mathcal{C}_2) \geq \frac{1-\delta}{d}\left(f\left(\frac{d\lambda}{1-2\delta}\right)-\delta\right)N^{d-2} \quad \mbox{and} \quad H<2 u \delta {\rm cap}(\mathcal{C}_2)\,.
\end{equation}
\end{lemma}
\begin{proof}
Suppose that the events $\mathcal{A}$ and $\mathcal{B}$ both happen.   Define a set $S$ as
\begin{equation*}
S=\cup_{B_x \in \mathcal{C}_1}(\mathcal{I}^u \cap B_x) \,.
\end{equation*}
Then, 
\begin{equation*}
\begin{split}
{\rm cap}(\mathcal{C}_1) &\;\,\overset{*}{=}\;\, \sum_{x \in \partial_i B(0,N)} e_{B(0,N)}(x) P_x[H_{\mathcal{C}_1}<\infty] \\
&\;\, \leq\;\; \sum_{x \in \partial_i B(0,N)} e_{B(0,N)}(x) \frac{P_x[H_S<\infty]}{\min_{z \in \mathcal{C}_1}P_z[H_S<\infty]} \quad (\mbox{by strong Markov property}) \\
&\overset{\eqref{eq:condition1}}{\leq} \sum_{x \in \partial_i B(0,N)} e_{B(0,N)}(x) \frac{1}{1-\delta}P_x[H_S<\infty]=\frac{1}{1-\delta}{\rm cap}(S)  \,\overset{**}{\leq}\, \frac{\lambda}{1-\delta}N^{d-2} \,.
\end{split}
\end{equation*}
where we use Lemma 1.12 in \cite{DRS14} for the equality marked with $*$ and use the definition of the event $\mathcal{B}$ for the inequality marked with $**$. We write $\widetilde{\mathcal{C}}_1$ and $\widetilde{\mathcal{C}}_2$ for $\mathbbm{R}^d$-fillings of $\mathcal{C}_1$ and $\mathcal{C}_2$ respectively. 
By Lemma~\ref{lem:disc-cont} and the fact that $L >c_2(\delta)$, $K >c_3(\delta)$,
$$
\widetilde{{\rm cap}}(\widetilde{\mathcal{C}}_1) \leq \frac{d}{1-\delta} {\rm cap}(\mathcal{C}_1) \leq \frac{d\lambda}{1-2 \delta}N^{d-2}\,.
$$
By $\mathcal{A}$, we have that ${\rm Card}({\mathcal{C}}_1) + {\rm Card}({\mathcal{C}}_2) \geq (1-\rho) \big(\frac{2N}{(2K+1)L}\big)^d$.
Hence, by Proposition~\ref{prop:property-bc} and the fact that $L>c_4(\delta,K), \rho<c_5(\delta,K)$, we have
$$
\widetilde{{\rm cap}}(\widetilde{\mathcal{C}}_2) \geq \left(f\left(\frac{d\lambda}{1-2\delta}\right)-\delta\right)N^{d-2}\,.
$$
Thus, by Lemma \ref{lem:laplace-transf},
\begin{equation}
\label{eq:estimate C_2}
{\rm cap}(\mathcal{C}_2) \geq \frac{1-\delta}{d} \widetilde{{\rm cap}}(\widetilde{\mathcal{C}}_2) \geq \frac{1-\delta}{d}  \left(f\left(\frac{d\lambda}{1-2\delta}\right)-\delta\right)N^{d-2}\,.
\end{equation}

We now turn to the upper bound of $H$. By Lemma~\ref{lem:compare-eq} and $K \geq c_1(\delta)$, we have
\begin{align*}
H &\;\,=\;\, \sum_{x \in \mathcal{C}_2} L^u(x) e_{\mathcal{C}_2}(x) \leq \sum_{x \in \mathcal{C}_2} L^u(x) (1+\delta) \bar{e}_B(x) e_{\mathcal{C}_2}(B)\\
&\;\, =\;\, (1+\delta) \sum_{B \in \mathcal{C}_2} e_{\mathcal{C}_2}(B) \sum_{x \in B} L^u(x) \bar e_B(x) \\
&\overset{\eqref{eq:condition2}}{<} (1+\delta) \sum_{B \in \mathcal{C}_2} e_{\mathcal{C}_2}(B) \delta u  \leq 2\delta u {\rm cap}(\mathcal{C}_2)\,.
\end{align*}
This finishes the proof.
\end{proof}
We now come back to the proof of Theorem~\ref{thm:1-1}.
\begin{proof}[Proof of the upper bound in Theorem~\ref{thm:1-1}]
By Lemma~\ref{lem:laplace-transf}, for each choice of $\mathcal{C}_2$, if $\mathcal{A} \cap \mathcal{B}$ happens
$$
\mathbbm{P}[H<2\delta u  {\rm cap}(\mathcal{C}_2)] \leq e^{2u \sqrt{\delta} {\rm cap}(\mathcal{C}_2)}\mathbbm{E}\Big[e^{-\frac{1}{\sqrt{\delta}}H} \Big]=\exp \left(2u \sqrt{\delta}{\rm cap}(\mathcal{C}_2)- \frac{u}{1+\sqrt{\delta}} {\rm cap}(\mathcal{C}_2)\right)\,.
$$
This together with \eqref{eq:es-c2h} gives
\begin{align*}
\mathbbm{P}[\mathcal{A} \cap \mathcal{B}] \;\,& \leq\;\sum_{E} \mathbbm{P}\left[\mathcal{A} \cap \mathcal{B} \cap \{\mathcal{C}_2 = E \} \right]\\
&\leq\; 2^{\left(\frac{2N}{(2K+1)L}\right)^d}\exp\left(u \left[2 \sqrt{\delta} - \frac{1}{1+\sqrt{\delta}} \right]\frac{1-\delta}{d}\left(f\left( \frac{d \lambda}{1-2\delta} \right) - \delta \right)N^{d-2} \right) \,.
\end{align*}
where we sum over all possible choices of $\mathcal{C}_2$,
We know from \eqref{eq:def-L} that $(N/L)^d=o(N^{d-2})$. This combined with Proposition~\ref{prop:es-ac} shows that
\begin{equation*}
\limsup_{N \rightarrow \infty} \frac{1}{N^{d-2}} \log \mathbbm{P}[\mathcal{B}] \leq u \left[2 \sqrt{\delta} - \frac{1}{1+\sqrt{\delta}} \right] \frac{1-\delta}{d} \left(f\left( \frac{d \lambda}{1-2\delta} \right) - \delta \right) \,.
\end{equation*}
Let $\delta$ tend to zero. By Claim (4) in Proposition~\ref{prop:property-f}, i.e., the continuity of $f$
\begin{equation*}
\limsup_{N \rightarrow \infty} \frac{1}{N^{d-2}} \log \mathbbm{P}[ \mathcal{B}] \leq -  \frac{u}{d} f(d \lambda) \,.
\end{equation*}
This finishes the proof.\end{proof}

\begin{proof}[Proof of Theorem~\ref{thm:1-2}]
Noting \eqref{eq:1-2quick},
Theorem~\ref{thm:1-2} follows from Theorem~\ref{thm:1-1} and Varadhan's lemma (see e.g.\ Theorem 4.3.1 of \cite{DZ10}).
\end{proof}

\begin{remark}
\label{remark:local-time} 1) We can also prove Claim (1) in Theorem \ref{thm:1-2} through the approach in \cite{S19b}. (Although both approaches are based on coarse-graining strategies, the definitions of good boxes are different, leading to different types of estimates.) Furthermore, we can also prove an entropic repulsion result in terms of local time. 
For $x \in \mathbbm{Z}^d$, let $L^1_x$ (resp.\ $L^2_x$) be the local time of interlacements $\mathcal{I}^{u_1}_1$ (resp.\ $\mathcal{I}^{u_2}_2$) at the vertex $x$. Then, we can show that for $u_1>u_2>0$ and all $\epsilon>0$,
\begin{equation}
\label{eq:repulsion-lt}
\lim_{N \rightarrow \infty} \mathbbm{P}\Big{[}\frac{1}{N^d} \sum_{x \in B(0,N)} L^2_x <\epsilon \Big{|}\mathcal{I}^{u_1}_1 \cap \mathcal{I}^{u_2}_2 \cap B(0,N)= \emptyset\Big{]} =1 \,.
\end{equation}
We briefly discuss the proof strategy of \eqref{eq:repulsion-lt}. In this case, we call an $L$-box good if one of the two interlacements has a small averaged local time, or they have a positive fraction of intersection points. Under this definition, most $L$-boxes are still good with super-exponential probability. Then, under the non-intersection conditioning, we know that forcing the interlacements with smaller intensity to have small averaged local time is the strategy with the highest probability. Together with some calculations similar to those in Section~\ref{sec:5}, we obtain \eqref{eq:repulsion-lt} as desired.

We can also give a global characterization of local times under this conditioning. Define the local time profile $\mathscr{L}^1_N, \mathscr{L}^2_N$ by
$$
\mathscr{L}^1_N= \frac{1}{N^d} \sum_{x \in \mathbbm{Z}^d} L^1_x \delta_{\frac{x}{N}} \,;\qquad \mathscr{L}^2_N= \frac{1}{N^d} \sum_{x \in \mathbbm{Z}^d} L^2_x \delta_{\frac{x}{N}}.
$$
For $u_1>u_2>0$ and any $R \geq 1$, we can show that
\begin{equation}
\begin{split}
&\lim_{N \rightarrow \infty} \mathbbm{E}\Big{[}d_R(\mathscr{L}^1_N \,, u_1 ) \wedge 1  \Big{|}\mathcal{I}^{u_1}_1 \cap \mathcal{I}^{u_2}_2 \cap B(0,N)= \emptyset\Big{]} =0 \,,\mbox{ and } \\
&\lim_{N \rightarrow \infty} \mathbbm{E}\Big{[}d_R(\mathscr{L}^2_N\,, u_2 \widetilde{e}_{\widetilde{B}(0,1)}(\cdot)^2) \wedge 1  \Big{|}\mathcal{I}^{u_1}_1 \cap \mathcal{I}^{u_2}_2 \cap B(0,N)= \emptyset\Big{]} =0\,,
\end{split}
\end{equation}
where $d_R$ is the $1$-Wasserstein distance confined in the box $\widetilde{B}(0,R)$. These entropic repulsion results can be obtained through the approach in \cite{CN20a}. For brevity, we will not include the proof here. One last comment is that Claim (2) in Theorem~\ref{thm:1-2} and \eqref{eq:repulsion-lt} cannot imply each other, and we choose to include the proof of the former since it naively follows from Theorem~\ref{thm:1-1} and does not require extra calculations.

\noindent 2) We now discuss the case where $u_1 = u_2$. Claim (1) in Theorem~\ref{thm:1-2} still holds in this case, but we cannot get the entropic repulsion result, i.e.\ Claim (2), naively from Theorem~\ref{thm:1-1}. However, similar to \eqref{eq:repulsion-lt}  we can show that one of the interlacements will have small local time densities in the macroscopic box and by symmetry we know that this probability is asymptotically one half, i.e., for $u_1 = u_2$ and all $\epsilon>0$,
\begin{equation}
\lim_{N \rightarrow \infty}\mathbbm{P}\Big{[}\frac{1}{N^d} \sum_{x \in B(0,N)} L^2_x <\epsilon \big{|}\mathcal{I}^{u_1}_1 \cap \mathcal{I}^{u_2}_2 \cap B(0,N)= \emptyset\Big{]} =\frac{1}{2}\,.
\end{equation}
\end{remark}

\begin{acks}[Acknowledgments]
We thank Bruno Schapira and Jiajun Tong for inspiring discussions. A large part of this work was done when the second author was an undergraduate at Peking University.
\end{acks}

\begin{funding}
The first author is supported by the National Key R\&D Program of China (No.\ 2020YFA0712900 and No.\ 2021YFA1002700) and NSFC (No.\ 12071012).
\end{funding}

\bibliography{Interlacements}
\bibliographystyle{imsart-nameyear}

\end{document}